\documentclass[12pt]{article}

\usepackage[table]{xcolor}
\usepackage{amssymb,amsmath}
\usepackage[margin=0.75in]{geometry}
\usepackage{mathrsfs}
\usepackage[all]{xy}
\usepackage{tikz}
\usepackage{tkz-graph}
\usepackage{tkz-berge}
\usetikzlibrary{decorations.pathreplacing}
\usepackage{diagbox}
\usepackage{pdflscape}
\usepackage{hyperref}

\title{Linear chord diagrams with long chords}
\author{Everett Sullivan}
\date{Version 1.0 \today}

\newtheorem{theorem}{Theorem} 

\newtheorem{definition}[theorem]{Definition}
\newtheorem{example}{Example}

\newtheorem{lemma}[theorem]{Lemma}

\numberwithin{equation}{section}

\def\qed{\hfill {\hbox{${\vcenter{\vbox{             
   \hrule height 0.4pt\hbox{\vrule width 0.4pt height 6pt
   \kern5pt\vrule width 0.4pt}\hrule height 0.4pt}}}$}}}

\newenvironment{proof}[1][Proof]{\smallskip\noindent{\bf #1.}\quad}
{\qed\par\medskip}

\newcommand{\abs}[1]{\lvert {#1} \rvert}

\newcommand{\para}[1]{\left( {#1} \right)}

\newcommand{\sumover}[3]{\sum_{{#1}}^{{#2}}{{#3}}}

\newcommand{\Image}{\text{Im}}

\newcommand{\restricto}[2]{\left.{#1}\right|_{{#2}}}

\newcommand\drawArc[2]{
	\draw (#1 - 1,0) -- (#2/2 + #1/2 - 1,#2/2 - #1/2) -- (#2 - 1,0);
}

\newcommand\drawDiagram[2]{
	\draw (0,0) -- (#1 - 1,0);
	\foreach \x/\y in {#2} {\drawArc{\x}{\y}};
	\foreach \z in {1,...,#1} {\node [below] at (\z-1,0) {\z};}
}

\newcommand{\minimalClass}[2]{\mathcal{M}^{({#1})}_{{#2}}}

\begin{document}

\maketitle
	
	\begin{abstract}
		A linear chord diagram of size $n$ is a partition of the set $\{1,2,\cdots,2n\}$ into sets of size two, called chords.
		From a table showing the number of linear chord diagrams of degree $n$ such that every chord has length at least $k$,
		we observe that if we proceed far enough along the diagonals, they are given by a geometric sequence.
		We prove that this holds for all diagonals, and identify when the effect starts.
	\end{abstract}
	
	\section{Introduction}
	
		A linear chord diagram is a matching of $\{1,2,\cdots,2n\}$.
		Chord diagrams arise in many different contexts, from the study of RNA ~\cite{Reidys11} to knot theory ~\cite{Chmutov12}. In combinatorics, chord diagrams show up in the m\'{e}nage problem ~\cite{Lucas91}, partitions ~\cite{Hsieh73}, and interval orders ~\cite{Winkler90}. 
		This paper will address diagrams where there is a specified minimum length for each chord.
		From a table counting the number of such diagrams for $n$ and $k$, we observe that if we proceed far enough along the diagonals, they are given by a geometric sequence.
		We prove that this holds for all diagonals, and identify when the effect starts.

	\section{Statement of Result}

	A \textit{linear chord diagram} of \textit{size} $n$ is a partition of the set $\{1,2,\cdots,2n\}$ into parts of size 2.
	
	We can draw linear chord diagrams with arcs connecting the partition blocks.
	
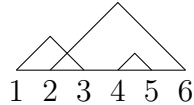
\begin{center}
		\begin{tikzpicture}[scale=.45]
			\drawDiagram{6}{1/3,2/6,4/5}
		\end{tikzpicture}
	\end{center}
	
	If $c = \{s_{c},e_{c}\}$ where $s_{c} < e_{c}$ is a block of a linear chord diagram.
	We say that $s_{c}$ is the \textit{start point} of $c$ and $e_{c}$ is the \textit{end point}.
	Then \textit{length} of $c$ is $e_{c} - s_{c}$.
	
	We say that a chord $c$ \textit{covers} $i$ if $s_{c} < i < e_{c}$.
	We say that a chord $c$ \textit{covers} a chord $d$ if it covers $s_{d}$ and $e_{d}$.
	
	\begin{definition}
		Let $D_{n}$ denote the set of all linear chord diagrams with $n$ chords.
		
		Let $\minimalClass{k}{}$ denote the class of all linear chord diagrams such that every chord has length at least k.
		
		Let $\minimalClass{k}{n}$ denote the set of all linear chord diagrams with $n$ such that every chord has length at least k.
	\end{definition}
	
	Table \ref{tab:data1} shows the sizes of $\minimalClass{k}{n}$ for various $n$ and $k$.
	If $k$ is fixed, $\minimalClass{n}{k}$ can be computed using on the order of $2^{k}n^{2}$ arithmetic operations.
	$a_{n} = \abs{\minimalClass{n}{2}}$ and $b_{n} = \abs{\minimalClass{n}{3}}$ can be computed using linear recurrences:
	\begin{align*}
		a_{n} & = (2n-1)a_{n-1} + a_{n-2} \\
		b_{n} & = (2n+2)b_{n-1} - (6n-10)b_{n-2} + (6n-16)b_{n-3} - (2n-8)b_{n-4} - b_{n-5}.
	\end{align*}
	The recurrence for $\abs{\minimalClass{n}{2}}$, can be found in ~\cite{Hazewinkel95}; the recurrence for $\abs{\minimalClass{n}{3}}$ is new.
	Conjecturally, there are linear recurrences for every $\minimalClass{n}{k}$ where $k$ is fixed: We will address these matters elsewhere.
	
	{
	\begin{table}
	\caption {Counting chord diagram with long chords} \label{tab:data1} 
	\centering
	\resizebox{\columnwidth}{!}{%
	\begin{tabular}{|c|c|c|c|c|c|c|c|c|c|c|c|}
		\hline
		$n$ & 1 & 2 & 3 & 4 & 5 & 6 & 7 & 8 & 9 & 10 & 11 \\ \hline
		$\abs{\minimalClass{n}{1}}$ & 1 & 3 & 15 & 105 & 945 & 10395 & 135135 & 2027025 & 34459425 & 654729075 & 13749310575 \\ \hline
		$\abs{\minimalClass{n}{2}}$ & 0 & 1 & \cellcolor{black!20!white} 5 & 36 & 329 & 3655 & 47844 & 721315 & 12310199 & 234615096 & 4939227215 \\ \hline
		$\abs{\minimalClass{n}{3}}$ & 0 & 0 & 1 & \cellcolor{black!20!white} 10 & 99 & 1146 & 15422 & 237135 & 4106680 & 79154927 & 1681383864 \\ \hline
		$\abs{\minimalClass{n}{4}}$ & 0 & 0 & 0 & 1 & \cellcolor{black!20!white} 20 & \cellcolor{black!20!white} 292 & 4317 & 69862 & 1251584 & 24728326 & 535333713  \\ \hline
		$\abs{\minimalClass{n}{5}}$ & 0 & 0 & 0 & 0 & 1 & \cellcolor{black!20!white} 40 & \cellcolor{black!20!white} 876 & 16924 & 332507 & 6944594 & 156127796 \\ \hline
		$\abs{\minimalClass{n}{6}}$ & 0 & 0 & 0 & 0 & 0 & 1 & \cellcolor{black!20!white} 80 & \cellcolor{black!20!white} 2628 & \cellcolor{black!20!white} 67404 & 1627252 & 39892549 \\ \hline
		$\abs{\minimalClass{n}{7}}$ & 0 & 0 & 0 & 0 & 0 & 0 & 1 & \cellcolor{black!20!white} 160 & \cellcolor{black!20!white} 7884 & \cellcolor{black!20!white} 269616 & 8075052  \\ \hline
		$\abs{\minimalClass{n}{8}}$ & 0 & 0 & 0 & 0 & 0 & 0 & 0 & 1 & \cellcolor{black!20!white} 320 & \cellcolor{black!20!white} 23652 & \cellcolor{black!20!white} 1078464  \\ \hline
		$\abs{\minimalClass{n}{9}}$ & 0 & 0 & 0 & 0 & 0 & 0 & 0 & 0 & 1 & \cellcolor{black!20!white} 640 & \cellcolor{black!20!white} 70956  \\ \hline
		$\abs{\minimalClass{n}{10}}$ & 0 & 0 & 0 & 0 & 0 & 0 & 0 & 0 & 0 & 1 & \cellcolor{black!20!white} 1280  \\ \hline
		$\abs{\minimalClass{n}{11}}$ & 0 & 0 & 0 & 0 & 0 & 0 & 0 & 0 & 0 & 0 & 1 \\ \hline
	\end{tabular}
	}
	
	\bigskip
	
	The first four rows can be found in the OEIS under the identification numbers A001147, A000806, A190823, and A190824, respectively.
	\end{table}
	}
	
	Here we address the diagonals of the table. The shaded squares highlight a pattern.
	The number in the square one below and one to the right, is exactly $(n-k+1)$ our current square.
	This pattern holds for all such squares,
	
	\begin{theorem}\label{thm:main}
	
		Let $n$ and $k$ be positive integers such that $n \geq 3(n - k)$ and $n \geq k$.
		Then $\abs{\minimalClass{k+1}{n+1}} = (n-k+1)\abs{\minimalClass{k}{n}}$.
	
	\end{theorem}
	
	\section{Outline of the proof}
	
	We consider each diagonal separately.
	We refer to the $i^{\text{th}}$ diagonal as all the entries such that $(n - k + 1) = i$.
	For any entry $\minimalClass{k}{n}$ the $i^{th}$ diagonal we create $(n - k + 1)$ functions $\alpha_{n,k,j}$ ($j \in \{0,\cdots,n-k\}$) which are injective into $\minimalClass{k+1}{n+1}$.
	
	We show that the images of these functions are disjoint and cover $\minimalClass{k+1}{n+1}$.
	And so there are $(n-k+1)$-times as many elements in $\minimalClass{k+1}{n+1}$ as there are in $\minimalClass{k}{n}$.
	
	To create the bijection $\alpha_{n,k,j}$ we consider the middle $2(n-k)$ indices.
	
	Here is an example from an element of $\minimalClass{4}{6}$
	
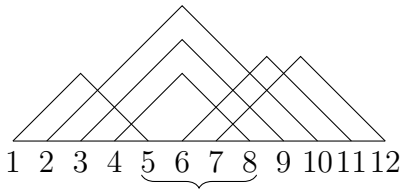
	\begin{center}
			\begin{tikzpicture}[scale=.45]
				\draw (0,0) -- (11,0);
				\draw(0,0) -- (2,2) -- (4,0);
				\draw(1,0) -- (5,4) -- (9,0);
				\draw(2,0) -- (5,3) -- (8,0);
				\draw(3,0) -- (5,2) -- (7,0);
				\draw(5,0) -- (7.5,2.5) -- (10,0);
				\draw(6,0) -- (8.5,2.5) -- (11,0);
				\node [below] at (0,0) {1};
				\node [below] at (1,0) {2};
				\node [below] at (2,0) {3};
				\node [below] at (3,0) {4};
				\node [below] at (4,0) {5};
				\node [below] at (5,0) {6};
				\node [below] at (6,0) {7};
				\node [below] at (7,0) {8};
				\node [below] at (8,0) {9};
				\node [below] at (9,0) {10};
				\node [below] at (10,0) {11};
				\node [below] at (11,0) {12};
				\draw [decorate,decoration={brace,amplitude=5pt,mirror},xshift=0pt,yshift=0pt]
(3.8,-1) -- (7.2,-1) node [black,midway,yshift=-.5cm]
{};
			\end{tikzpicture}
		\end{center}
		Any chords starting or ending in the middle indices are highlighted
		
		
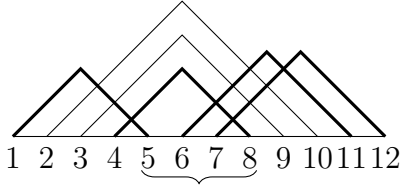
\begin{center}
			\begin{tikzpicture}[scale=.45]
				\draw (0,0) -- (11,0);
				\draw[very thick](0,0) -- (2,2) -- (4,0);
				\draw(1,0) -- (5,4) -- (9,0);
				\draw(2,0) -- (5,3) -- (8,0);
				\draw[very thick](3,0) -- (5,2) -- (7,0);
				\draw[very thick](5,0) -- (7.5,2.5) -- (10,0);
				\draw[very thick](6,0) -- (8.5,2.5) -- (11,0);
				\node [below] at (0,0) {1};
				\node [below] at (1,0) {2};
				\node [below] at (2,0) {3};
				\node [below] at (3,0) {4};
				\node [below] at (4,0) {5};
				\node [below] at (5,0) {6};
				\node [below] at (6,0) {7};
				\node [below] at (7,0) {8};
				\node [below] at (8,0) {9};
				\node [below] at (9,0) {10};
				\node [below] at (10,0) {11};
				\node [below] at (11,0) {12};
				\draw [decorate,decoration={brace,amplitude=5pt,mirror},xshift=0pt,yshift=0pt]
(3.8,-1) -- (7.2,-1) node [black,midway,yshift=-.5cm]
{};
			\end{tikzpicture}
		\end{center}
		
		A new chord is inserted covering only the indices in the middle
		
		
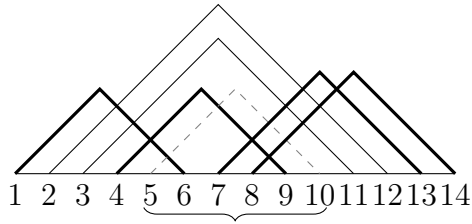
\begin{center}
			\begin{tikzpicture}[scale=.45]
				\draw (0,0) -- (13,0);
				\draw[very thick](0,0) -- (2.5,2.5) -- (5,0);
				\draw(1,0) -- (6,5) -- (11,0);
				\draw(2,0) -- (6,4) -- (10,0);
				\draw[very thick](3,0) -- (5.5,2.5) -- (8,0);
				\draw[gray, dashed](4,0) -- (6.5,2.5) -- (9,0);
				\draw[very thick](6,0) -- (9,3) -- (12,0);
				\draw[very thick](7,0) -- (10,3) -- (13,0);
				\node [below] at (0,0) {1};
				\node [below] at (1,0) {2};
				\node [below] at (2,0) {3};
				\node [below] at (3,0) {4};
				\node [below] at (4,0) {5};
				\node [below] at (5,0) {6};
				\node [below] at (6,0) {7};
				\node [below] at (7,0) {8};
				\node [below] at (8,0) {9};
				\node [below] at (9,0) {10};
				\node [below] at (10,0) {11};
				\node [below] at (11,0) {12};
				\node [below] at (12,0) {13};
				\node [below] at (13,0) {14};
				\draw [decorate,decoration={brace,amplitude=5pt,mirror},xshift=0pt,yshift=0pt]
(3.8,-1) -- (9.2,-1) node [black,midway,yshift=-.5cm]
{};
			\end{tikzpicture}
		\end{center}
		
		The new chord then has its start point iteratively swapped with the starting points of the unbolded cords, starting with the one that started last and stopping when there are $j$ unswapped unbolded chords.
		
		
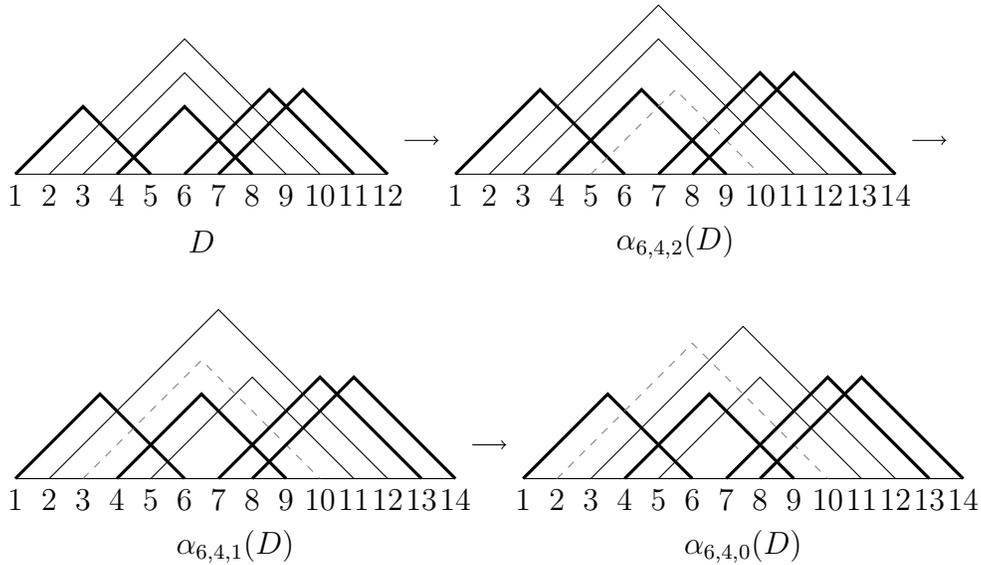
\begin{center}
			\begin{tikzpicture}[scale=0.45]
				\begin{scope}[shift={(0,-0.5)}]
					\draw (0,0) -- (11,0);
					\draw[very thick](0,0) -- (2,2) -- (4,0);
					\draw(1,0) -- (5,4) -- (9,0);
					\draw(2,0) -- (5,3) -- (8,0);
					\draw[very thick](3,0) -- (5,2) -- (7,0);
					\draw[very thick](5,0) -- (7.5,2.5) -- (10,0);
					\draw[very thick](6,0) -- (8.5,2.5) -- (11,0);
					\node [below] at (0,0) {1};
					\node [below] at (1,0) {2};
					\node [below] at (2,0) {3};
					\node [below] at (3,0) {4};
					\node [below] at (4,0) {5};
					\node [below] at (5,0) {6};
					\node [below] at (6,0) {7};
					\node [below] at (7,0) {8};
					\node [below] at (8,0) {9};
					\node [below] at (9,0) {10};
					\node [below] at (10,0) {11};
					\node [below] at (11,0) {12};
					\node at (5.5,-2) {$D$};
					\draw[->] (11.5,1) -- (12.5,1);
				\end{scope}
				\begin{scope}[shift={(13,-0.5)}]
					\draw (0,0) -- (13,0);
					\draw[very thick](0,0) -- (2.5,2.5) -- (5,0);
					\draw(1,0) -- (6,5) -- (11,0);
					\draw(2,0) -- (6,4) -- (10,0);
					\draw[very thick](3,0) -- (5.5,2.5) -- (8,0);
					\draw[gray, dashed](4,0) -- (6.5,2.5) -- (9,0);
					\draw[very thick](6,0) -- (9,3) -- (12,0);
					\draw[very thick](7,0) -- (10,3) -- (13,0);
					\node [below] at (0,0) {1};
					\node [below] at (1,0) {2};
					\node [below] at (2,0) {3};
					\node [below] at (3,0) {4};
					\node [below] at (4,0) {5};
					\node [below] at (5,0) {6};
					\node [below] at (6,0) {7};
					\node [below] at (7,0) {8};
					\node [below] at (8,0) {9};
					\node [below] at (9,0) {10};
					\node [below] at (10,0) {11};
					\node [below] at (11,0) {12};
					\node [below] at (12,0) {13};
					\node [below] at (13,0) {14};
					\node at (6.5,-2) {$\alpha_{6,4,2}(D)$};
					\draw[->] (13.5,1) -- (14.5,1);
				\end{scope}
				\begin{scope}[shift={(0,-9.5)}]
					\draw (0,0) -- (13,0);
					\draw[very thick](0,0) -- (2.5,2.5) -- (5,0);
					\draw(1,0) -- (6,5) -- (11,0);
					\draw[gray, dashed](2,0) -- (5.5,3.5) -- (9,0);
					\draw[very thick](3,0) -- (5.5,2.5) -- (8,0);
					\draw(4,0) -- (7,3) -- (10,0);
					\draw[very thick](6,0) -- (9,3) -- (12,0);
					\draw[very thick](7,0) -- (10,3) -- (13,0);
					\node [below] at (0,0) {1};
					\node [below] at (1,0) {2};
					\node [below] at (2,0) {3};
					\node [below] at (3,0) {4};
					\node [below] at (4,0) {5};
					\node [below] at (5,0) {6};
					\node [below] at (6,0) {7};
					\node [below] at (7,0) {8};
					\node [below] at (8,0) {9};
					\node [below] at (9,0) {10};
					\node [below] at (10,0) {11};
					\node [below] at (11,0) {12};
					\node [below] at (12,0) {13};
					\node [below] at (13,0) {14};
					\node at (6.5,-2) {$\alpha_{6,4,1}(D)$};
					\draw[->] (13.5,1) -- (14.5,1);
				\end{scope}
				\begin{scope}[shift={(15,-9.5)}]
					\draw (0,0) -- (13,0);
					\draw[very thick](0,0) -- (2.5,2.5) -- (5,0);
					\draw[gray, dashed](1,0) -- (5,4) -- (9,0);
					\draw(2,0) -- (6.5,4.5) -- (11,0);
					\draw[very thick](3,0) -- (5.5,2.5) -- (8,0);
					\draw(4,0) -- (7,3) -- (10,0);
					\draw[very thick](6,0) -- (9,3) -- (12,0);
					\draw[very thick](7,0) -- (10,3) -- (13,0);
					\node [below] at (0,0) {1};
					\node [below] at (1,0) {2};
					\node [below] at (2,0) {3};
					\node [below] at (3,0) {4};
					\node [below] at (4,0) {5};
					\node [below] at (5,0) {6};
					\node [below] at (6,0) {7};
					\node [below] at (7,0) {8};
					\node [below] at (8,0) {9};
					\node [below] at (9,0) {10};
					\node [below] at (10,0) {11};
					\node [below] at (11,0) {12};
					\node [below] at (12,0) {13};
					\node [below] at (13,0) {14};
					\node at (6.5,-2) {$\alpha_{6,4,0}(D)$};
				\end{scope}
			\end{tikzpicture}
		\end{center}

	\section{Details of the proof}
	
	\begin{definition}
	
		Let $C$ be a linear chord diagram, then we define $L_{n,k} = \{1,2,\cdots,k\}$, $M_{n,k} = \{k+1,k+3,\cdots,2n-k\}$, and $R_{n,k} = \{2n-k+1,2n-k+1,\cdots,2n\}$.
		Let $C_{n,k}$ denote the set of all chords $c \in C$ such that $s_{c} \in M_{n,k}$ or $e_{c} \in M_{n,k}$, and $S_{C}$ denote the set of all chords $c \in C$ such that $c \notin C_{n,k}$. 
	
	\end{definition}

	\begin{lemma}\label{lem:noChord}
		Given any linear chord diagram in $\minimalClass{k}{n}$ such that $n \geq 3(n-k)$ and $n \geq k$, there is no chord $c$ such that $s_{c},e_{c} \in M_{n,k}$.
	\end{lemma}
	
	\begin{proof}
		If a chord has both its start point and end point inside $M_{n,k}$, then the largest length it could have, is when it starts at $k+1$ and ends at $2n-k$.
		So the maximum length any such chord could have is $2n-2k-1$.
		But $n \geq 3(n-k)$ which is equivalent to $3k \geq 2n$.
		Thus the maximum length any such chord could have is $2n - 2k -1 \leq 3k - 2k - 1 = k-1$.
		But every chord must have length at least $k$.
		Thus there is no chord such that its indices of the start point and end point lie inside $M_{n,k}$
	\end{proof}

	\begin{lemma}\label{lem:sameNumber}
		Given any linear chord diagram in $\minimalClass{k}{n}$ such that $n \geq 3(n-k)$ and $n \geq k$, $C_{n,k}$ contains exactly $n-k$ chords that start in $M_{n,k}$ and $n-k$ chords that end in $M_{n,k}$.
	\end{lemma}
	
	\begin{proof}
	
		We first observe that no chord has its end index in $L_{n,k}$, since it if did, its maximum length would be $k-1$.
		Similarly, no chord has its start index in $R_{n,k}$ since it if did, its maximum length would be $2n-(2n-k+1) = k-1$.
		Thus every index in $L_{n,k}$ is a start index, and every index in $R_{n,k}$ is an end index.
		We also observe that $\abs{L_{n,k}} = \abs{R_{n,k}}$.
		
		Consider all chords in $S_{c}$.
		Since they neither start nor end in $M_{n,k}$, they must start in $L_{n,k}$ and end in $R_{n,k}$.
		
		Thus $\abs{L_{n,k}} - \abs{S_{C}}$ chords start in $L_{n,k}$ and end in $M_{n,k}$, and $\abs{R_{n,k}} - \abs{S_{C}}$ chords end in $R_{n,k}$ and start in $M_{n,k}$.
		
		By Lemma \ref{lem:noChord}, every Chord in $M$ either start in $L_{n,k}$ or ends in $R_{n,k}$.
		
		Thus $M_{n,k}$ has the same number of start indices as end indices, and that number is $n-k$.
	
	\end{proof}

	\begin{lemma}\label{lem:left}
		Given any linear chord diagram $C \in \minimalClass{k+1}{n+1}$ such that $n \geq 3(n-k)$ and $n \geq k$, let $a$ be the chord whose end index is $2n-k+2$ (i.e. the smallest element in $R_{n+1,k+1}$).
		Let $m$ be the number of chords $b \in S_{C}$ such that $s_{b} < s_{a}$.
		Then $m < n - k + 1$.
	\end{lemma}
	
	\begin{proof}
	
		Let $M^{\ast}$ by the ordered set of all chords $c \in C_{n+1,k+1}$ in such that $e_{c} \in M$.
		We say $k < c$ for $k,c \in M^{\ast}$ if $e_{k} < e_{c}$.
		Observe that $M^{\ast}$ is completely ordered.
		By Lemma \ref{lem:sameNumber}, we have $\abs{M^{\ast}} = n-k$
		We may relabel the chords is $M^{\ast}$ to be $\{c_{1},c_{2},\cdots, c_{n-k}\}$.
		Observe that by Lemma \ref{lem:sameNumber}, $e_{c_{i}} \leq (k+1) + (n-k) + i = n+i+1$.
		Since $\ell_{c_{i}} = e_{c_{i}} - s_{c_{i}} \geq k+1$ we have $s_{c_{i}} \leq n + i + 1 - (k+1) = n - k + i$.
		Let $m_{i}$ be the number of chords $a \in S_{C}$ such that $s_{a} < s_{c_{i}}$.
		Then $m_{1} < n-k + 1$.
		The largest number of start indices to the left of $s_{c_{2}}$ is $n-k + 1$, but if it were that large, one of them must be the start of $c_{1}$.
		Thus $m_{2} < n-k + 1$.
		By induction we have $m_{i} < n-k + 1$ for all $i$.
		
		Now suppose $m \geq n-k + 1$, then $s_{c_{i}} < s_{a}$ for all $i$ since otherwise $m_{i} \geq n-k + 1$.
		Thus $s_{a} \geq (n-k+1) + (n-k) + 1 = 2n-2k + 2$
		Thus $\ell_{a}$ is bounded above by $2n-k+2 - (2n-2k + 2) = k < k+1$.
		
		Thus $m < n-k + 1$.
	
	\end{proof}

	\begin{definition}
		We define $\alpha_{n,k,i}$ for $i \in \{0,\cdots,n-k\}$, $n \geq k$, and $n \geq 3(n-k)$ to be a map from $\minimalClass{k}{n}$ to $D_{n}$ as follows.
		Given a diagram $C$, we insert a new chord $c$ with start point right before $M_{n,k}$ and end point right after $M_{n,k}$ to get diagram $C^{\ast}$.
		We then swap the start index of the new chord with the closest start index of a chord in $S_{C}$ to its left.
		We continue to swap until there are $i$ start indices of chords in $S_{C}$ to its left.
		
		Observe that since $n \geq 3(n-k)$, that the number of chords in $S$ is at least $n - (2n-2k) \geq 3(n-k) - 2(n-k) = n-k$
		Thus every $\alpha$ exists and is well defined.
	\end{definition}
	
	\begin{example}
	
		Obtaining $C^{\ast}$ from $C$ is shown below
	
		
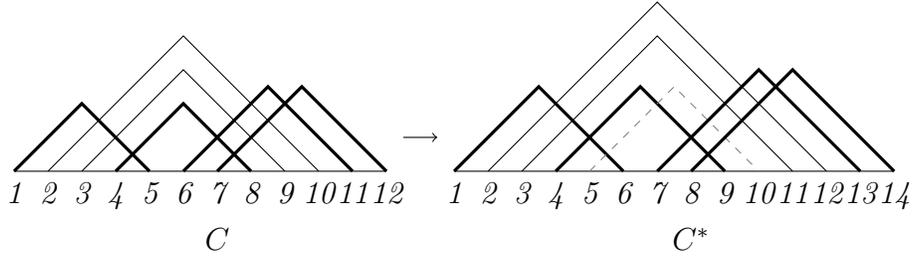
\begin{center}
			\begin{tikzpicture}[scale=.45]
				\begin{scope}[shift={(0,-0.5)}]
					\draw (0,0) -- (11,0);
					\draw[very thick](0,0) -- (2,2) -- (4,0);
					\draw(1,0) -- (5,4) -- (9,0);
					\draw(2,0) -- (5,3) -- (8,0);
					\draw[very thick](3,0) -- (5,2) -- (7,0);
					\draw[very thick](5,0) -- (7.5,2.5) -- (10,0);
					\draw[very thick](6,0) -- (8.5,2.5) -- (11,0);
					\node [below] at (0,0) {1};
					\node [below] at (1,0) {2};
					\node [below] at (2,0) {3};
					\node [below] at (3,0) {4};
					\node [below] at (4,0) {5};
					\node [below] at (5,0) {6};
					\node [below] at (6,0) {7};
					\node [below] at (7,0) {8};
					\node [below] at (8,0) {9};
					\node [below] at (9,0) {10};
					\node [below] at (10,0) {11};
					\node [below] at (11,0) {12};
					\node at (6,-2) {$C$};
					\draw[->] (11.5,1) -- (12.5,1);
				\end{scope}
				\begin{scope}[shift={(13,-0.5)}]
					\draw (0,0) -- (13,0);
					\draw[very thick](0,0) -- (2.5,2.5) -- (5,0);
					\draw(1,0) -- (6,5) -- (11,0);
					\draw(2,0) -- (6,4) -- (10,0);
					\draw[very thick](3,0) -- (5.5,2.5) -- (8,0);
					\draw[gray, dashed](4,0) -- (6.5,2.5) -- (9,0);
					\draw[very thick](6,0) -- (9,3) -- (12,0);
					\draw[very thick](7,0) -- (10,3) -- (13,0);
					\node [below] at (0,0) {1};
					\node [below] at (1,0) {2};
					\node [below] at (2,0) {3};
					\node [below] at (3,0) {4};
					\node [below] at (4,0) {5};
					\node [below] at (5,0) {6};
					\node [below] at (6,0) {7};
					\node [below] at (7,0) {8};
					\node [below] at (8,0) {9};
					\node [below] at (9,0) {10};
					\node [below] at (10,0) {11};
					\node [below] at (11,0) {12};
					\node [below] at (12,0) {13};
					\node [below] at (13,0) {14};
					\node at (7,-2) {$C^{\ast}$};
				\end{scope}
			\end{tikzpicture}
		\end{center}
		
		Here is $\alpha_{3,2,0}$ applied to an element of $\minimalClass{2}{3}$
		
		\begin{center}
			\begin{tikzpicture}[scale=.45]
				
				\begin{scope}[shift={(0,-.5)}]
					\draw (0,0) -- (5,0);
					\draw[very thick] (0,0) -- (1,1) -- (2,0);
					\draw (1,0) -- (2.5,1.5) -- (4,0);
					\draw[very thick] (3,0) -- (4,1) -- (5,0);
					\node [below] at (0,0) {1};
					\node [below] at (1,0) {2};
					\node [below] at (2,0) {3};
					\node [below] at (3,0) {4};
					\node [below] at (4,0) {5};
					\node [below] at (5,0) {6};
					\draw[->] (5.5,1) -- (7.5,1);
				\end{scope}
				
				\begin{scope}[shift={(8,-.5)}]
					\draw (0,0) -- (7,0);
					\draw[very thick] (0,0) -- (1.5,1.5) -- (3,0);
					\draw (1,0) -- (3.5,2.5) -- (6,0);
					\draw[gray, dashed] (2,0) -- (3.5,1.5) -- (5,0);
					\draw[very thick] (4,0) -- (5.5,1.5) -- (7,0);
					\node [below] at (0,0) {1};
					\node [below] at (1,0) {2};
					\node [below] at (2,0) {3};
					\node [below] at (3,0) {4};
					\node [below] at (4,0) {5};
					\node [below] at (5,0) {6};
					\node [below] at (6,0) {7};
					\node [below] at (7,0) {8};
					\draw[->] (7.5,1) -- (9.5,1);
				\end{scope}
				
				\begin{scope}[shift={(18,-.5)}]
					\draw (0,0) -- (7,0);
					\draw[very thick] (0,0) -- (1.5,1.5) -- (3,0);
					\draw (2,0) -- (4,2) -- (6,0);
					\draw[gray, dashed] (1,0) -- (3,2) -- (5,0);
					\draw[very thick] (4,0) -- (5.5,1.5) -- (7,0);
					\node [below] at (0,0) {1};
					\node [below] at (1,0) {2};
					\node [below] at (2,0) {3};
					\node [below] at (3,0) {4};
					\node [below] at (4,0) {5};
					\node [below] at (5,0) {6};
					\node [below] at (6,0) {7};
					\node [below] at (7,0) {8};
				\end{scope}
			\end{tikzpicture}
		\end{center}
		
		Here is $\alpha_{4,3,1}$ applied to an element of $\minimalClass{3}{4}$.
		
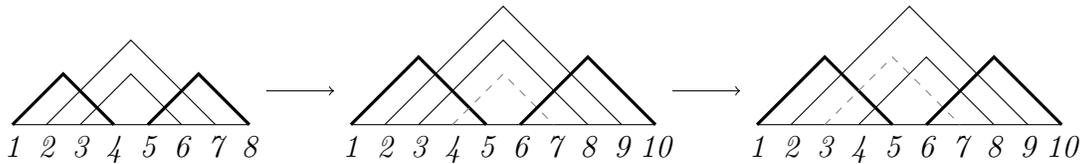
		\begin{center}
			\begin{tikzpicture}[scale=.45]
				
				\begin{scope}[shift={(0,-.5)}]
					\draw (0,0) -- (7,0);
					\draw[very thick] (0,0) -- (1.5,1.5) -- (3,0);
					\draw (1,0) -- (3.5,2.5) -- (6,0);
					\draw (2,0) -- (3.5,1.5) -- (5,0);
					\draw[very thick] (4,0) -- (5.5,1.5) -- (7,0);
					\node [below] at (0,0) {1};
					\node [below] at (1,0) {2};
					\node [below] at (2,0) {3};
					\node [below] at (3,0) {4};
					\node [below] at (4,0) {5};
					\node [below] at (5,0) {6};
					\node [below] at (6,0) {7};
					\node [below] at (7,0) {8};
					\draw[->] (7.5,1) -- (9.5,1);
				\end{scope}
				
				\begin{scope}[shift={(10,-.5)}]
					\draw (0,0) -- (9,0);
					\draw[very thick] (0,0) -- (2,2) -- (4,0);
					\draw (1,0) -- (4.5,3.5) -- (8,0);
					\draw (2,0) -- (4.5,2.5) -- (7,0);
					\draw[gray, dashed] (3,0) -- (4.5,1.5) -- (6,0);
					\draw[very thick] (5,0) -- (7,2) -- (9,0);
					\node [below] at (0,0) {1};
					\node [below] at (1,0) {2};
					\node [below] at (2,0) {3};
					\node [below] at (3,0) {4};
					\node [below] at (4,0) {5};
					\node [below] at (5,0) {6};
					\node [below] at (6,0) {7};
					\node [below] at (7,0) {8};
					\node [below] at (8,0) {9};
					\node [below] at (9,0) {10};
					\draw[->] (9.5,1) -- (11.5,1);
				\end{scope}
				
				\begin{scope}[shift={(22,-.5)}]
					\draw (0,0) -- (9,0);
					\draw[very thick] (0,0) -- (2,2) -- (4,0);
					\draw (1,0) -- (4.5,3.5) -- (8,0);
					\draw[gray, dashed] (2,0) -- (4,2) -- (6,0);
					\draw (3,0) -- (5,2) -- (7,0);
					\draw[very thick] (5,0) -- (7,2) -- (9,0);
					\node [below] at (0,0) {1};
					\node [below] at (1,0) {2};
					\node [below] at (2,0) {3};
					\node [below] at (3,0) {4};
					\node [below] at (4,0) {5};
					\node [below] at (5,0) {6};
					\node [below] at (6,0) {7};
					\node [below] at (7,0) {8};
					\node [below] at (8,0) {9};
					\node [below] at (9,0) {10};
				\end{scope}
			\end{tikzpicture}
		\end{center}
		
		where the thick lines are chords in $C_{n,k}$, the thin chords are in $S_{C}$ and the greyed dashed chord is the new inserted one.
	
	\end{example}

	\begin{definition}
		We define $\beta_{n,k}$ for $n \geq k$, and $n > 3(n-k)$ to be a map from $\minimalClass{k}{n}$ to $D_{n-1}$ as follows.
		Given a diagram $C$, we denote $c$ to be the cord with end point right after $M_{n,k}$.
		We then swap the start index of the new chord with the closest start index of a chord in $S_{C}$ to its right.
		We continue to swap until there are no more start indices of chords in $S_{C}$ to its right.
		We then remove chord $c$.
	\end{definition}

	\begin{lemma}\label{lem:alpha}
		$\alpha_{n,k,i}\para{\minimalClass{k}{n}} \subseteq \minimalClass{k+1}{n+1}$.
	\end{lemma}
	
	\begin{proof}
		We see that the result will have $n+1$ chords, so it suffices to show that every chord has length at least $k+1$.
		
		Consider a chord $c$ in $C_{n,k}$, it either has $s_{c} \in M_{n,k}$ and $e_{c} \in R_{n,k}$, in which case it length is increased by 1, since we inserted a index between $_{n,k}M$ and $R_{n,k}$.
		Or $e_{c} \in M_{n,k}$ and $s_{c} \in L_{n,k}$, in which case it length is increased by 1, since we inserted a index between $M_{n,k}$ and $L_{n,k}$.
		Since the length of such a chord had to be at least $k$ to begin with, it must have at least length $k+1$ after applying $\alpha$.
		
		Consider the chord we just inserted.
		It will cover all the indices in $M_{n,k}$, and every time we swap, another index will be covered.
		Since there are a total of $n$ chords before inserting, of which $M_{n,k}$ contains $2n-2k$ of them, and it swaps until there are $i$ chord to its left in $S$, it swapped with at least $n-(2n-2k) - i$.
		Recall that the length of the chord will be the number of indices it covers plus 1.
		Thus its length is at least $1 +(2n-2k) + (n-(2n-2k)) - i = 1 + n - i \geq 1 + n - (n-k) = k+1$.
		As desired.
		
		Now consider chords in $S_{C}$.
		
		There are two cases, either it has its start index swapped at some point or it didn't.
		If it didn't, then it covers the new chord $c$, and has length greater then $c$'s length.
		Thus the chord has length at least $k+1$ as desired.
		
		If it did swap, then either its starting index increased by 1 or more.
		
		Suppose that its starting index increased by 1.
		Then the number of indices that lie in between its endpoints has increased by 1.
		When we inserted $c$, it was increased by 2, but then we moved the starting index forward by 1, causing it to lose 1.
		Thus its length increased by exactly 1.
		Since it must of have length $k$ to begin, with, in now has length at least $k+1$.
		
		Suppose that its starting index increased by more then 1.
		Let $a$ be its original starting index after inserting $c$ and $b$ be its starting index after inserting and swapping $c$.
		Then the index $b-1$ is the starting index of some point in $M_{n+1,k+1}$, since  $b-a > 1$ and otherwise $b$ would have occurred sooner.
		Thus the the chord with starting index $b-1$ has length at least $k+1$.
		Since the ending index of our chord lies in $R$ which is at least 1 more then the ending index of the chord at $b-1$, the length of our chord after swapping is at least $k+1$.
		
		Thus $\alpha_{n,k,i}\para{\minimalClass{k}{n}} \subseteq \minimalClass{k+1}{n+1}$ as desired.
	\end{proof}

	\begin{lemma}\label{lem:beta}
		$\beta_{n,k}\para{\minimalClass{k}{n}} \subseteq \minimalClass{k-1}{n-1}$.
	\end{lemma}
	
	\begin{proof}
		We see that the result will have $n-1$ chords, so it suffices to show that every chord has length at least $k-1$.
		
		Consider a chord $r$ in $C_{n,k}$, it either has $s_{r} \in M_{n,k}$ and $e_{r} \in R_{n,k}$, in which case it length is decreased by 1, since we removed the first index in $R_{n,k}$.
		Or $e_{r} \in M$ and $s_{r} \in L$, in which case it length is decreased by 1, since we removed the last index from $L_{n,k}$.
		Since the length of such a chord had to be at least $k$ to begin with, it must have at least length $k-1$ after applying $\beta_{n,k}$.
		
		Consider a chord $r$ in $S_{C}$
		We break it into two cases:
		
		Case 1: $s_{r}$ was swapped with $s_{c}$ at some point.
		Then $s_{r}$ has decreased by at least 1, which means that $\ell_{r}$ increased by at least 1.
		But when we remove $s_{c}$ a the end, $\ell_{r}$ is deceased by 2.
		Thus $\ell_{r}$ never deceases by more then 1.
		Since $\ell_{r} = k$, the length of $r$ must be at least length $k-1$ after applying $\beta_{n,k}$.
		
		Case 2: $s_{r}$ did not swap with $s_{c}$ at some point.
		Then $s_{r} < s_{c}$, which means that, $\ell_{r}$ is at least $2 + \ell_{c} = k+2$ since $\ell_{c}$ has length at least $k$.
		When we remove $s_{c}$ a the end, $\ell_{r}$ is deceased by 2.
		Thus $\ell_{r}$ never deceases by more then 2.
		Since $\ell_{r} \geq k+2$, the length of $r$ must be at least length $k$ after applying $\beta_{n,k}$.
		
		Thus $\beta_{n,k}\para{\minimalClass{k}{n}} \subseteq \minimalClass{k-1}{n-1}$ as desired.
	\end{proof}

	\begin{proof}[Proof (of theorem \ref{thm:main})]
	
		We shall proceed by constructing $(n-k+1)$ injective function from $\minimalClass{k}{n}$ to $\minimalClass{k+1}{n+}$ such that their images partition $\minimalClass{k+1}{n+1}$.
		Let $C \in \minimalClass{k}{n}$
		
		Let $E_{n,k,i}$ be the set of all linear chord diagrams in $\minimalClass{k}{n}$ such that the chord $c$ with $e_{s} = 2n-k+1$ (i.e. the first index after $M_{n,k}$) has $i$ start points of chords in $S_{C}$ to its left.
		
		Then by lemma \ref{lem:left} the collection $\{E_{n+1,k+1,0},\cdots,E_{n+1,k+1,n-k-1}\}$ partitions $\minimalClass{k+1}{n+1}$.
		By construction we see that $\Image(\alpha_{n,k,i}) \subseteq E_{n+1,k+1,i}$.
		We also see that both $\restricto{\beta_{n+1,k+1}}{E_{n+1,k+1,i}} \circ \alpha_{n,k,i}$ and $\alpha_{n,k,i} \circ \restricto{\beta_{n+1,k+1}}{E_{n+1,k+1,i}}$ are the identity map.
		Thus there is a bijection between $\minimalClass{k}{n}$ and $E_{n,k,i}$ for every $i$.
		
		Thus
		\begin{displaymath}
			\abs{\minimalClass{k+1}{n+1}} = \sumover{i=0}{n-k}{\abs{\alpha_{n,k,i}\para{\minimalClass{k}{n}}}} = (n-k+1)\abs{\minimalClass{k}{n}}
		\end{displaymath}
		As desired.
	
	\end{proof}
	
\bibliography{ChordDiagonalPaper}
\bibliographystyle{abbrv}

\end{document}